\documentclass{proc-l}

\usepackage{amssymb}

\newcommand{\Span}{\mathop{\rm span}\,}

\newcommand{\RR}{\mathbb{R}}
\newcommand{\NN}{\mathbb{N}}

\newcommand{\TT}{\mathbb{T}}

\newcommand{\Ab}{{\boldsymbol{A}}}
\newcommand{\Bb}{{\boldsymbol{B}}}

\newcommand{\Qb}{{\boldsymbol{Q}}}

\newcommand{\Ub}{{\boldsymbol{U}}}
\newcommand{\Vb}{{\boldsymbol{V}}}
\newcommand{\Wb}{{\boldsymbol{W}}}

\newcommand{\eb}{{\boldsymbol{e}}}

\newcommand{\ub}{{\boldsymbol{u}}}
\newcommand{\vb}{{\boldsymbol{v}}}
\newcommand{\wb}{{\boldsymbol{w}}}

\newcommand{\zb}{{\boldsymbol{z}}}

\newcommand{\epsb}{{\boldsymbol{\epsilon}}}

\newtheorem{theorem}{Theorem}[section]
\newtheorem{lemma}[theorem]{Lemma}
\newtheorem{corollary}[theorem]{Corollary}

\theoremstyle{definition}
\newtheorem{definition}[theorem]{Definition}

\theoremstyle{remark}
\newtheorem{remark}[theorem]{Remark}

\numberwithin{equation}{section}

\begin{document}
\title{Decomposition of tensors}
\author{Juan--Manuel Pe{\~n}a$^*$}
\address{Depto. Matem\'{a}tica Aplicada, Fac. Ciencias,
  Universidad de Zaragoza, E--50009 Zaragoza, SPAIN.}
\curraddr{}
\email{jmpena@unizar.es}
\thanks{$^*$Partially supported by the
Spanish Research Grant MTM2012-31544 and by the Gobierno de
Arag\'on.}

\author{Tomas Sauer}
\address{Lehrstuhl Mathematik mit Schwerpunkt Digitale
  Bildverarbeitung, Universit\"at Passau, Innstr. 43, 94032 Passau,
  GERMANY}
\curraddr{}
\email{Tomas.Sauer@uni-passau.de}
\thanks{}

%    \subjclass is required.

\date{}

\dedicatory{}

%    "Communicated by" -- provide editor's name; required.
\commby{}

\begin{abstract}
  We consider representations of tensors as sums of decomposable
  tensors or, equivalently, decomposition of multilinear forms into
  one--forms. In this short note we show that there exists a
  particular finite strongly orthogonal decomposition which is essentially
  unique and yields all critical points of the multilinear form on the
  torus. In particular, this determines exactly the number of critical
  points of the multilinear form, giving an affirmative answer to a
  finiteness conjecture by Friedland.
\end{abstract}

\maketitle
% \noindent
% \textbf{Keywords:} Subdivision, Hermite, Taylor expansion,
% Factorization.

% \noindent
% \textbf{AMS Subject classification:} 41A60, 65D15, 13P05.

\section{Introduction}
A real \emph{$p$--tensor} $\Ab \in \RR^{n_1 \times \cdots \times n_p}$,
$n_1,\dots,n_p \in \NN$, is a
multidimensional array defined by coefficients 
\begin{equation}
  \label{eq:TensCoeffDef}
  a_\alpha \in \RR, \qquad \alpha \le \nu := ( n_1,\dots,n_p ).
\end{equation}
Following our previous work \cite{LampingPenaSauer13P}, we find it
most convenient to view the tensor $\Ab$ as a \emph{multilinear
  form} $\Ab : \RR^\nu \to \RR$, defined as
$$
\Ab [\ub] := \Ab [ \ub_1,\dots,\ub_p ] = \sum_{\alpha \le \nu} a_\alpha \,
u_{1,\alpha_1} \cdots u_{p,\alpha_p}, \qquad \ub = ( \ub_1,\dots,\ub_p
) \in \RR^\nu,
$$
where $\RR^\nu$ is a structured version of $\RR^n$, $n := n_1
\cdots n_p$.

A \emph{decomposable tensor} or \emph{one--form} $\Wb = \wb_1 \otimes
\cdots \otimes \wb_p$, $\wb_j \in \RR^{n_j}$, is the particular
multilinear form given by
$$
w_\alpha = \prod_{j=1}^p w_{j,\alpha_j}, \qquad \alpha \le \nu.
$$
We will use the scalar product
$$
\left\langle \Ab,\Bb \right\rangle := \sum_{\alpha \le \nu} a_\alpha
\, b_\alpha
$$
which has the well--known useful property that
\begin{equation}
  \label{eq:SkProd1}
  \left\langle \Ab,\wb_1 \otimes \cdots \otimes \wb_p \right\rangle =
  \Ab [ \wb_1,\dots,\wb_p ].
\end{equation}
A multivector $\ub \in \RR^\nu$ is called \emph{critical} for a multilinear
form $\Ab$ if there exists some $\lambda \in \RR$ such that
\begin{equation}
  \label{eq:CritPtDef}
  \Ab [ \ub_1,\dots,\ub_{j-1},\ub,\ub_{j+1},\dots,\ub_p ] = \lambda \,
  \ub_j^T \ub, \qquad \ub \in \RR^{n_j}.
\end{equation}
Any location where the multilinear form $\Ab$ has a extremum on
$\TT^\nu := \{ \ub \in \RR^n : \| \ub_j \|_2 = 1,\, j=1,\dots,p \}$
has to be a critical point as follows by considering the Lagrange
multipliers for the restricted optimization problem
$$
\max \left\{ \Ab [\ub] : \ub \in \TT^\nu \right\}.
$$
Two multivectors $\ub,\vb \in \RR^\nu$ are called \emph{orthogonal} if the
associated one--forms are orthogonal, i.e., if
\begin{equation}
  \label{eq:OrthoDef}
  0 = \left\langle \Ub, \Vb \right\rangle := \left\langle \ub_1
    \otimes \cdots \otimes \ub_p, \vb_1 \otimes \cdots \otimes \vb_p
  \right\rangle = \prod_{j=1}^p \ub_j^T \vb_j.
\end{equation}
Throughout this paper we will consistently use the notation $\ub =
(\ub_1,\dots,\ub_p)$ for the multivector in $\RR^\nu$ and $\Ub =
\ub_1 \otimes \cdots \otimes \ub_p$ for the one--form generated by
$\ub$. $\ub$ and $\vb$ will be called \emph{strongly
  orthogonal}, see
\cite{kolda01:_orthog,leibovici98:_singul_value_decom_way_array}, if
\begin{equation}
  \label{eq:StrongOrtho}
  \langle \Ub,\Vb \rangle = 0 \qquad \mbox{and} \qquad \ub_j^T \vb_j
  \in \{-1,0,1\}, \quad j=1,\dots,p.
\end{equation}
We will denote this relationship by $\ub \perp \vb$. Vectors $\epsb =
( \epsilon_1,\dots,\epsilon_p ) \in \{ -1,1 \}^p$ are called
\emph{sign distributions}. A sign distribution is called \emph{even} if
$\epsilon_1 \cdots \epsilon_p = 1$ and \emph{odd} if $\epsilon_1
\cdots \epsilon_p = -1$. It is worthwhile to recall the well--known
fact that one--forms and multilinear forms are only defined up to even
sign distributions:
$$
\Ab [ \ub ] = \Ab [ \epsb \ub ] = \Ab [ \epsilon_1
\ub_1,\dots,\epsilon_p \ub_p ], \qquad
\wb_1 \otimes \cdots \otimes \wb_p = ( \epsilon_1 \wb_1) \otimes
\cdots \otimes ( \epsilon_p \wb_p).
$$
This ambiguity of multilinearity will show up later.

\section{Decompositions}
Our goal is to study decompositions of a multilinear form $\Ab$ into a sum of
normalized one--forms, that is,
\begin{equation}
  \label{eq:ADecomp}
  \Ab = \sum_{k=1}^r \sigma_k \, \wb_1^k \otimes \cdots \otimes \wb_p^k
  = \sum_{k=1}^r \sigma_k \, \Wb_k, \qquad \sigma_1 \ge \cdots \ge
  \sigma_r > 0, 
\end{equation}
where we always will request that the components of the one--forms are
normalized, that is
\begin{equation}
  \label{eq:ADecomp2}
  \| \wb_j^k \| = 1, \qquad k=1,\dots,r, \, j=1,\dots,r.  
\end{equation}
The nonnegativity of the $\sigma_k$ requested in \eqref{eq:ADecomp}
can always be ensured by applying an odd sign distribution 
An \emph{orthogonal
  decomposition} is a decomposition of the form \eqref{eq:ADecomp}
where $\langle \Wb_j,\Wb_k \rangle = 0$, $0 \le j < k \le r$, while a
\emph{strongly orthogonal decomposition} (SOD) has to satisfy
\begin{equation}
  \label{eq:SODDef}
  \wb^j \perp \wb^k, \quad 0 \le j < k \le r, \qquad \wb^k := (
  \wb_1^k, \dots, \wb_p^k ).
\end{equation}
SODs can be constructed very easily and in various ways. To that end,
let $\Qb_j \in \RR^{n_j \times n_j}$, $j=1,\dots,p$, be orthogonal
matrices and define the $n$ strongly orthogonal vectors
$$
\wb_\alpha = \left( \Qb_1 \eb_{\alpha_1}, \dots, \Qb_p \eb_{\alpha_p}
\right), \qquad \alpha \le \nu.
$$
Then
$$
\Ab = \sum_{\alpha \le \nu} \langle \Ab,\Wb_\alpha \rangle \,
\Wb_\alpha = \sum_{\alpha \le \nu} \Ab [ \wb_\alpha ] \, \Wb_\alpha
$$
is an SOD after modifying some of the $\wb_\alpha$ by means of an odd
sign distribution of $\Ab [ \wb_\alpha ] < 0$ and arranging the sum
with respect to the size of $\Ab [ \wb_\alpha ]$. Of course, the
number $r$ of nonzero terms, though smaller than $n$, depends on the
chosen basis matrices $\Qb_j$.

Conversely, any SOD \eqref{eq:ADecomp} defines orthogonal
matrices. Indeed, since the $\wb_j^k$ either coincide up to sign or
are orthogonal, the set $\{ \wb_j^k : k=1,\dots,r \}$ contains at most
$n_j$ orthonormal elements up to sign. Completing it to an orthonormal
basis if necessary and arranging these vectors as columns of an
orthogonal matrix $\Qb_j$, we have that there exist numbers $\alpha_j
(k)$ such that $\wb_j^k = \pm \Qb_j \eb_{\alpha_j (k)}$,
$k=1,\dots,r$. Hence, in the above notation, $\wb^k =
\epsb_k \wb_{\alpha(k)}$, $k=1,\dots,r$, where $\epsb_k$ is an even
sign distribution. In other words, any SOD can be completed to a
strongly orthogonal basis of all multilinear forms.

As mentioned before, the number of nonzero terms in a SOD depends on
the choice of the $\Qb_j$ and is bounded from above by $n$. The
optimal SODs, of course, would be those with a minimal number of terms
and this number describes the complexity of the multilinear form.

\begin{definition}
  The minimal number $r$ such that $\Ab$ has a strongly orthogonal
  decomposition with $r$ terms is called the (strong) \emph{rank} of
  $\Ab$, denoted as $r(\Ab)$.
\end{definition}

\noindent
A particular way to obtain a strongly orthogonal decomposition is by
means of the \emph{greedy strongly orthogonal decomposition} (GSOD)
that computes an SOD in the following way:
\begin{enumerate}
\item Compute $\sigma_1$ and $\wb_1$ as
  \begin{equation}
    \label{eq:GOSD1}
    \sigma_1 = \Ab [\wb^1] = \max \left\{ \Ab [\ub] : \ub \in \TT^\nu
    \right\}.    
  \end{equation}
\item For $k=1,2,\dots$ compute
  \begin{equation}
    \label{eq:GOSD2}
    \sigma_{k+1} = \Ab [\wb^{k+1}] =  \max \left\{ \Ab [\ub] : \ub \in \TT^\nu,
      \ub \perp \wb^j, \, j=1,\dots,k \right\}    
  \end{equation}
  until $\sigma_{k+1} = 0$.
\end{enumerate}
Even if this is not so easy to do computationally and quite tricky
from a numerical point
of view, see e.g. \cite{friedland13,kolda01:_orthog,zhang01:_rank}, it
serves its purpose as a theoretical tool. Since the above process
terminates after at most $n$ steps, there always exists a GSOD for any
multilinear form $\Ab$.

As shown in \cite{LampingPenaSauer13P}, critical points and SODs
interact.

\begin{theorem}[\cite{LampingPenaSauer13P}, Theorem~20]\label{T:CritRep}
  If \eqref{eq:ADecomp} is an SOD of $\Ab$ and $\ub \in \RR^\nu$ is a critical
  point for $\Ab$, then $\Ub \in \Span \{ \Wb_1,\dots,\Wb_r\}$, i.e.
  \begin{equation}
    \label{eq:TCritRep}
    \Ub = \sum_{k=1}^r \langle \Ub,\Wb_k \rangle \, \Wb_k.  
  \end{equation}
\end{theorem}

\noindent
We also recall a characterization of critical points from
\cite{LampingPenaSauer13P} which we will reprove for the sake of
completeness. To that end we define the gradient components
$$
\zb_j = \zb_j (\ub) := \sum_{\ell=1}^p \Ab [
  \ub_1,\dots,\ub_{j-1},\eb_\ell,\ub_{j+1},\dots,\ub_p ] \, \eb_\ell,
  \qquad j=1,\dots,p,
$$
so that $\ub$ is critical iff $\ub = \sigma \zb$ for some $\sigma \in
\RR$. Substituting \eqref{eq:ADecomp}, we get that
\begin{equation}
  \label{eq:zjEq}
  \zb_j = \sum_{k=1}^r \sigma_k \, \left( \prod_{\ell \neq j}
    \ub_\ell^T \wb_\ell^k \right) \, \wb_j^k,
\end{equation}
from which we can derive the following reformulation of another result from
\cite{LampingPenaSauer13P}.

\begin{lemma}[\cite{LampingPenaSauer13P},Proposition~21]\label{L:LPSLemma}
  A component $\wb^k$ of an SOD \eqref{eq:ADecomp} for $\Ab$ is
  critical for $\Ab$ if and only if
  \begin{align*}
    \lefteqn{\{ \wb^1,\dots,\wb^r \}} \\
    & \cap \left\{
      ( \vb_1, \pm \wb_2^k,\dots, \pm \wb_p^k),\dots,( \pm
      \wb_1^k,\dots, \pm \wb^k_{p-1},\vb_p) :
      \begin{array}{c}
        \vb_j^T \wb_j^k = 0 \\ k=1,\dots,r
      \end{array}
    \right\} = \emptyset.    
  \end{align*}
\end{lemma}

\noindent
In other words, $\wb^k$ is critical if and only if no strongly
orthogonal $\wb \perp \wb^k$ can be found among $\wb^1,\dots,\wb^r$
that coincides with $\wb^k$ in $p-1$ components up to sign. We give the simple
proof of the lemma for the sake of completeness.

\begin{proof}
  By \eqref{eq:zjEq},
  the identity $\zb_j = \sigma \wb^k_j$ is equivalent to
  $$
  \sigma \wb_j^k = \sigma_k \wb_j^k + \sum_{\{ \ell : \wb_i^\ell = \pm
    \wb^k_i, \, i \neq j \}} \sigma_\ell \epsilon_\ell \, \wb_j^\ell,
  \qquad \epsilon_\ell = \prod_{i \neq j} (\wb_i^\ell)^T \wb_i^k \in \{ -1,1 \},
  $$
  and since $\wb^\ell \perp \wb^k$ for any $\ell$ in the sum on the
  right hand side, this sum is orthogonal to $\wb_j^k$ and thus
  vanishes if and only if $\wb^k$ is critical.
\end{proof}

\noindent
Applying Lemma~\ref{L:LPSLemma} to \eqref{eq:ADecomp} we get obtain
the following reformulation of the lemma.

\begin{corollary}\label{C:LPSLemma}
  A component $\wb_k$ of an SOD \eqref{eq:ADecomp} for $\Ab$ is
  critical for $\Ab$ if and only if
  \begin{equation}
    \label{eq:LPSLemma}
    \Ab \left[ \wb^k_1,\dots,\wb^k_{j-1},\wb,\wb^k_{j+1},\dots,\wb^k_p
    \right] = 0, \qquad \wb^T \wb_j^k = 0,\, \quad j=1,\dots,p.
  \end{equation}
\end{corollary}

\section{The GSOD and critical points}
In this section, we collect some properties of the GSOD. The most
crucial one is given in the following lemma.

\begin{lemma}\label{L:GSODCritLemma}
  Any component of a GSOD \eqref{eq:ADecomp} of $\Ab$ is critical for $\Ab$.
\end{lemma}

\begin{proof}
  We will show by induction on $k$ that $\wb^1,\dots,\wb^k$ are
  critical, which is obvious for $k=1$ as $\wb^1$ is the location of a
  global maximum of $\Ab$ on $\TT^\nu$ and thus clearly critical.

  Suppose now that the result has been proved for some $k \ge 1$ and
  that $\sigma_{k+1}$ as defined in \eqref{eq:GOSD2} is positive. We
  first note that the optimization problem can be written as
  \begin{equation}
    \label{eq:GSODCritLemmaPf1}
    \max \Ab [\ub], \quad \left\{
      \begin{array}{rclcl}
        0 & = & ( \ub_j^T \wb_j^\ell ) \left( ( \ub_j^T \wb_j^\ell )^2
          - 1 \right), & \quad & j=1,\dots,p,\, \ell = 1,\dots,k, \\
        0 & = & \displaystyle{\prod_{j=1}^p \ub_j^T \wb_j^\ell}, & &
        \ell=1,\dots,k, \\
        0 & = & \ub_j^T \ub_j - 1, & & j=1,\dots,p.
      \end{array}
    \right.
  \end{equation}
  Since $\frac{\partial \Ab}{\partial \ub_j} (\ub) = \zb_j (\ub)$, $j
  =1,\dots,p$, there must exist Lagrange multipliers
  $\lambda_{j\ell}$, $\lambda_\ell$ and $\mu_j$, $\ell=1,\dots,k$,
  $j=1,\dots,p$, such that
  \begin{equation}
    \label{eq:GSODCritLemmaPf2}
    \zb_j (\wb^{k+1}) = \sum_{\ell=1}^k \left( \lambda_{j\ell} \left( 3 \left(
          (\wb_j^\ell)^T \wb_j^{k+1} \right)^2 - 1 \right) +
      \lambda_\ell \, \prod_{i \neq j} (\wb_i^\ell)^T
        \wb_i^{k+1} \right) \wb_j^\ell + 2 \mu_j \wb_j^{k+1} 
  \end{equation}
  holds for $j=1,\dots,p$. Comparing this with \eqref{eq:zjEq}, it
  follows that
  \begin{equation}
    \label{eq:GSODCritLemmaPf3}
    \zb_j (\wb^{k+1}) = \sigma_{k+1} \wb_j^{k+1} + \sum_{\ell = 1}^k \sigma_\ell
    \theta_\ell \wb_j^\ell, \qquad \theta_\ell := \prod_{i \neq j}
    (\wb_i^\ell) \wb_i^{k+1}.
  \end{equation}
  If $\theta_\ell \neq 0$ for some $\ell \in \{
  1,\dots,k \}$, we get $\wb_i^{k+1} = \epsilon_i \wb_i^\ell$,
  $\epsilon_i \in \{-1,1\}$, $i
  \neq j$, hence 
  \begin{eqnarray*}
    0 < \sigma_{k+1} & = & \Ab [ \wb^{k+1} ] = \Ab \left[
      \wb_1^{k+1},\dots,\wb_{j-1}^{k+1},\wb_j^{k+1}, \wb_{j+1}^{k+1},
      \dots,\wb_p^{k+1} \right] \\
    & = & \Ab \left[
      \epsilon_1 \, \wb_1^\ell,\dots, \epsilon_{j-1} \,
      \wb_{j-1}^\ell,\wb_j^{k+1}, \epsilon_{j+1} \, \wb_{j+1}^\ell,
      \dots,\epsilon_p \, \wb_p^\ell \right] = 0
  \end{eqnarray*}
  by Corollary~\ref{C:LPSLemma}. Since this is a contradiction, it
  follows that $\theta_\ell = 0$, $\ell = 1,\dots,k$, hence $\zb_j
  (\wb^{k+1}) = \sigma_{k+1} \wb_j^{k+1}$. In other words, $\wb^{k+1}$
  is critical, too, thus advancing the induction hypothesis and
  completing the proof.
\end{proof}

\begin{corollary}\label{C:GOSDMinimal}
  Any GOSD \eqref{eq:ADecomp} of $\Ab$ is \emph{minimal}, i.e., $r =
  r(\Ab)$. In particular, the GOSD can be used to compute $r(\Ab)$.
\end{corollary}

\begin{proof}
  Let
  $$
  \Ab = \sum_{k=1}^{r(\Ab)} \sigma_k^* \, \Wb_k^*
  $$
  be any minimal SOD. Since any of the strongly orthogonal, hence
  linearly independent components $\Wb_\ell$, $\ell = 1,\dots,r$ of a GOSD is
  contained in $\Span \{ \Wb_k^* : k \in r(A) \}$ by
  Theorem~\ref{T:CritRep}, it follows that $r (\Ab) \ge r$ while
  minimality yields $r(\Ab) \le r$, hence $r = r(\Ab)$.
\end{proof}

\begin{definition}
  A decomposition \eqref{eq:ADecomp} is called \emph{critical} if all
  its components are critical for $\Ab$. 
\end{definition}

\noindent
Note that a GSOD is critical by Lemma~\ref{L:GSODCritLemma}. We next
show that critical SODs can only generate a very limited number of
one--forms.

\begin{lemma}\label{L:CritDecomOne}
  Suppose that \eqref{eq:ADecomp} is a critical SOD of $\Ab$ and that
  $\ub \in \TT^\nu$ is such that $\Ub \in \Span \{ \Wb_k : k=1,\dots,r \}$. Then
  there exist $k \in \{ 1,\dots,r \}$ and an even sign distribution
  $\epsb$ such that $\ub = \epsb \wb^k$.
\end{lemma}

\begin{proof}
  Suppose that
  $$
  0 \neq \Ub = \sum_{k=1}^r c_k \, \Wb_k, \qquad c_k \in \RR,
  $$
  and let $k \in \{1,\dots,r\}$ be such that $c_k \neq 0$. Then
  $$
  0 \neq c_k = \langle \Ub, \Wb_k \rangle = \Ub [ \wb_k ] = \prod_{j=1}^p
  \ub_j^T  \wb_j^k,
  $$
  which implies that $\ub_j^T  \wb_j^k \neq 0$, $j=1,\dots,p$. For $j
  \in \{ 1,\dots,p \}$ and any
  $$
  \vb \in \{ \wb_j^\ell : \ell \neq j
  \} \cup \{ \wb_j^\ell : \ell=1,\dots,r \}^\perp, \qquad \vb^T
  \wb^k_j = 0,
  \qquad \| \vb \|_2 = 1,
  $$
  the vector
  $$
  \wb = \left( \wb_1^k,\dots,\wb_{j-1}^k,\vb,\wb_{j+1}^k,\dots,\wb_p^k
  \right)
  $$
  cannot belong to $\{ \epsb_\ell \wb^\ell : \epsb_\ell \in
  \{-1,1\}^p, \ell=1,\dots,r \}$ since
  $\wb^k$ is critical, see Lemma~\ref{L:LPSLemma}. Hence, $\langle
  \Wb,\Wb_\ell \rangle = 0$, $\ell = 1,\dots,r$, and therefore
  \begin{eqnarray*}
    0 & = & \sum_{\ell=1}^r c_\ell \langle \Wb_\ell,\Wb \rangle =
    \langle \Ub,\Wb \rangle
    = \Ub [ \wb ]
    = \prod_{\ell=1}^p \ub_\ell^T \wb_\ell = \ub_j^T \vb \,
    \prod_{\ell \neq j} \ub_\ell^T \wb_\ell^k,
  \end{eqnarray*}
  which implies that $\ub_j^T \vb = 0$ whenever $\vb^T \wb_j^k =
  0$. Since $\| \ub_j \|_2 = \| \wb_j^k \|_2 = 1$ this implies that
  $\ub_j = \pm \wb_j^k$ and proves the claim.
\end{proof}

\noindent
This result has some immediate consequences. First, we can now speak
of \emph{the} GSOD.

\begin{corollary}\label{C:GSODUnique}
  The GSOD is unique up to reordering and even sign distributions.
\end{corollary}

\begin{proof} 
  Let
  $$
  \Ab = \sum_{k=1}^r \sigma_k \Wb_k = \sum_{k=1}^{\widehat r} \widehat
  \sigma_k \widehat \Wb_k
  $$
  be two GSODs for $\Ab$. Then $r = \widehat r = r(\Ab)$ by
  Corollary~\ref{C:GOSDMinimal} and since for $k=1,\dots,r$ the
  multivector $\wb^k$ is critical for 
  $\Ab$ by Lemma~\ref{L:GSODCritLemma}, Lemma~\ref{L:CritDecomOne}
  yields that $\wb^k = \epsb \widehat \wb^\ell$ for some even $\epsb
  \in \{-1,1\}$ and $\ell \in \{ 1,\dots,r \}$. After modifying the
  $\widehat \Wb_k$ by
  these sign distributions, we thus have that $\{ \Wb_1,\dots,\Wb_r \}
  \subseteq \{ \widehat \Wb_1,\dots, \widehat \Wb_r \}$ and by
  symmetry the same argument yields also the inverse inclusion and
  proves the claim.
\end{proof}

\begin{corollary}\label{C:BestRank1}
  If (\ref{eq:ADecomp}) is a GSOD of $\Ab$, then the set of all best
  rank--one approximation of $\Ab$ consist, up to even sign
  distributions, exactly of the one--forms 
  generated by 
  $ \left\{ \sigma_1 \Wb_k : \sigma_k = \sigma_1 \right\}$.
\end{corollary}

\begin{proof}
  If $\sigma \Wb$ is a best rank--one approximation, then $\wb$ must
  be critical, hence there must be $k \in \{1,\dots,r\}$ and an even
  sign distribution $\epsb$ such that $\wb = \epsb \wb^k$. Since then
  $\Ab [\wb] = \Ab[\wb^k] = \sigma_k$, this is maximal if and only if
  $\sigma_k = \sigma_1 = \max \{ \Ab[\ub] : \ub \in \TT^\nu
  \}$. Conversely, we have for any sign distribution $\epsb$ and any
  $k$ with $\sigma_k = \sigma_1$ that $\Ab [\wb^k] = \sigma_1$.
\end{proof}

\begin{corollary}
  The best rank--one approximation is unique if and only if $\sigma_1
  > \sigma_2$ in the GSOD (\ref{eq:ADecomp}).
\end{corollary}

\begin{remark}
  If, as in \eqref{eq:ADecomp}, we require the \emph{singular values}
  $\sigma_k = \Ab [ \wb^k ]$ to be arranged in descending order, the
  rearrangement will only take place when two or more of these
  singular values coincide like in the simple example
  $$
  \Ab = \vb \otimes \vb \otimes \vb + \vb \otimes \wb \otimes \wb + \wb \otimes
  \vb \otimes \wb + \wb \otimes \wb \otimes \vb, \qquad  \vb,\wb \in \RR^2,
  \, \wb^T \vb = 0,
  $$
  of a critical SOD in $\RR^{(2,2,2)}$ from
  \cite{LampingPenaSauer13P}. Obviously, each of the four summands is
  a best rank--one approximation.
\end{remark}

\noindent
The next observation confirms a conjecture by
Friedland \cite{friedland13:_best}, mentioned in \cite{friedland13}.

\begin{corollary}\label{C:FriedConfirm}
  A multilinear form $\Ab$ has exactly $2^p \, r(\Ab)$ critical points.
\end{corollary}

\begin{proof}
  By Lemma~\ref{L:GSODCritLemma} and Lemma~\ref{L:CritDecomOne} a
  point $\ub \in \TT^\nu$ is critical if and only if there exist some
  $k \in \{ 1,\dots,r \}$ and a sign distribution $\epsb \in \{-1,1\}^p$
  such that $\ub = \epsb \wb^k$ which gives exactly $2^p r(\Ab)$
  critical points.
\end{proof}

\begin{remark}
  Friedland's conjecture is originally concerned about the number of
  fixpoints of a certain map, but since there is a one-to-one
  relationship between these fixpoints and the critical points of a
  multilinear function, Corollary~\ref{C:FriedConfirm} can be applied.
\end{remark}

\begin{remark}
  Since $\Ab [ \epsb \wb^k ] = \pm \sigma_k$ half of the critical
  points belong to maxima computed in the process of determining the
  GSOD, the other half locates the respective minima.
\end{remark}

\bibliographystyle{amsplain}
\bibliography{../bibls/cagd,../bibls/wavelets,../bibls/approx,../bibls/books,../bibls/interpol,../bibls/subdiv,../bibls/numerik}

\end{document}